\documentclass[twoside]{article}
\usepackage{amsmath, amsthm, amscd, amsfonts, amssymb, graphicx}
\usepackage{enumerate}
\usepackage[colorlinks=true, linkcolor=blue, urlcolor=cyan, citecolor=red]{hyperref}
\theoremstyle{definition}
\newtheorem{theorem}{Theorem}[section]
\newtheorem{lemma}[theorem]{Lemma}
\newtheorem{corollary}[theorem]{Corollary}
\newtheorem{definition}[theorem]{Definition}
\newtheorem{example}[theorem]{Example}

\theoremstyle{remark}
\newtheorem{remark}[theorem]{Remark}

\numberwithin{equation}{section}

\title{\bf On Reversing Operator Choi-Davis-Jensen inequality}
\author{$^{1}$Seyyed Saeid Hashemi Karouei\\
$^{2}$Mohammad Sadegh Asgari$^*$\\$^{3}$Mohsen Shah Hosseini}
\date{\small $^{1, 2}$Department of Mathematics, Faculty of Science,
Central Tehran Branch,\\ Islamic Azad University, Tehran, Iran.\\
$^{3}$Department of Mathematics, Shahr-e-Qods Branch,\\
Islamic Azad University, Tehran, Iran.\\ $^{1}$seyyed.s.hashemi@gmail.com\\
$^{2}$moh.asgari@iauctb.ac.ir\\
$^{3}$mohsen\_shahhosseini@yahoo.com }

\newcommand{\subject}[1]{\begin{flushleft}
\textbf{2010 AMS Subject Classification}: #1\end{flushleft}}
\newcommand{\keyword}[1]{\par\noindent \textbf{Keywords:} #1}

\newcommand{\eval}[2][\right]{\relax
\ifx#1\right\relax \left.\fi#2#1\rvert}

\renewcommand{\sectionmark}[1]{}
\begin{document}
\maketitle

\begin{abstract}
\noindent
In this paper, we first provide a better estimate of the second inequality in Hermite-Hadamard inequality.
Next, we study the reverse of the celebrated Davis-Choi-Jensen's inequality. Our results are employed to
establish a new bound for the operator Kantorovich inequality.

\vspace{.3cm}
\keyword{Hermite-Hadamard inequality; Davis-Choi-Jensen inequality; convex function;
self-adjoint operator; positive operator.}
\subject{Primary: 47A12; 47A30 Secondary: 26D15.}
\end{abstract}

\section{Introduction and Preliminaries}

A very interesting inequality for convex functions that has been widely studied in the literature is due to Hermit and Hadamard.
It provides a two-sided estimate of the mean value of a convex function. The Hermite-Hadamard inequality (for briefly, H-H inequality)
has several applications in the nonlinear analysis and geometry of Banach spaces, see \cite{bar}. Over the past decades, several
interesting generalizations, specific cases, and formulations of this remarkable inequality have been obtained for some different frameworks. In fact, it provides a necessary and sufficient condition for a function $f$ to be convex. Many well-known inequalities
can be obtained using the concept of convex functions. For details, interested readers can refer to \cite{r3, r4, r1, r2}.
This remarkable result of Hermit and Hadamard is as follows:
\vskip 0.4 true cm
\noindent{\bf Theorem A.} Let $f:J\to \mathbb{R}$ be a convex function, where $a,b\in J$ with $a<b$. Then
\begin{align}\label{eq1} f\big(\frac{a+b}{2}\big)\le \frac{1}{b-a}\int_a^b f(z)dz\le \frac{f(a)+f(b)}{2}.\end{align}
A history of this inequality can be found in \cite{k2}. An overview of the generalities and various developments can be found in
\cite {k1}.\par  In this paper, we first provide a better estimate of the second inequality in H-H inequality (\ref{eq1}).\par
Throughout this paper $\mathcal{H}$ and $\mathcal{K}$ are complex Hilbert spaces, and
$\mathcal{B}(\mathcal{H})$ denotes the algebra of all (bounded linear) operators on $\mathcal{H}$. Recall that an operator
$A$ on $\mathcal{H}$ is said to be positive (in symbol: $A\ge 0$) if $\langle Ax, x \rangle\ge 0$ for all $x \in \mathcal{H}$.
We write $A>0$ if $A$ is positive and invertible. For self-adjoint operators $A$ and $B$, we write $A\ge B$ if $A-B$ is positive,
i.e., $\langle Ax, x \rangle\ge\langle Bx, x\rangle$ for all $x \in\mathcal{H}$. We call it the usual order. In particular, for
some scalars $m$ and $M$, we write $m\le A\le M$ if
\begin{align*} m\langle x, x\rangle\le\langle Ax, x\rangle\le M\langle x, x \rangle, \hspace{1cm}\forall\;
x \in\mathcal{H}. \end{align*}
We extensively use the continuous functional calculus for self-adjoint operators, e.g., see \cite[p. 3]{book}.
\begin{definition}\label{def1.1}
A continuous function $f$ defined on the interval $J$ is called an operator convex function if
\begin{align*} f\big((1-v)A+vB \big)\le(1-v)f(A)+vf( B)\end{align*}
for every $0<v<1$ and for every pair of bounded self-adjoint operators $A$ and $B$ whose spectra are both in $J$.
\end{definition}
\begin{definition}\label{def1.2}
A linear map $\Phi: \mathcal{B}(\mathcal{H})\to \mathcal{B}(\mathcal{K})$ is positive if
$\Phi \left( A \right)$ is positive for all positive $A$ in $\mathcal{B}\left( \mathcal{H} \right)$. It is said to be
unital (or, normalized) if $\Phi \left( \mathbf{1}_\mathcal{H} \right)=\mathbf{1}_\mathcal{K}$.\end{definition}
We recall the Davis-Choi-Jensen inequality \cite{choi, davis} for operator convex functions, which is regarded as a
noncommutative version of Jensen's inequality:
\vskip 0.4 true cm
\noindent{\bf Theorem B.} Let $A\in \mathcal{B}\left( \mathcal{H} \right)$ be a self-adjoint operator with the spectra
contained in the interval $J$ and $\Phi $ be a positive unital linear map from $\mathbb{B}\left( \mathcal{H} \right)$ to
$\mathbb{B}\left( \mathcal{K} \right)$. If $f$ is operator convex function on an interval $J$, then \begin{equation}\label{cdj}
f\left( \Phi \left( A \right) \right)\le \Phi \left( f\left( A \right) \right).
\end{equation}
\indent Though in the case of convex function the inequality \eqref{cdj} does not hold in general, we have the following estimate
from \cite[Remark 4.14]{micic}.
\vskip 0.4 true cm
\noindent{\bf Theorem C.} Let $A\in \mathcal{B}\left( \mathcal{H} \right)$ be a self-adjoint operator with
$Sp\left( A \right)\subseteq \left[ m,M \right]$ for some scalars $m,M$ with $m<M$ and $\Phi $ be a positive unital
linear map from $\mathbb{B}(\mathcal{H})$ to $\mathbb{B}(\mathcal{K})$. If $f$ is non-negative
convex function, then \begin{equation}\label{eq2}
\frac{1}{\alpha}\Phi\big(f(A)\big)\le f\big(\Phi(A)\big)\le\alpha\Phi\big(f(A)\big),\end{equation}
where $\alpha$ is defined by
\begin{align*}\alpha =\max_{m\le t\le M}\Big\{\frac{1}{f(t)}\Big(\frac{M-t}{M+m}f(m)+\frac{t-m}{M+m}f(M)\Big)
\Big\}.\end{align*}

For other generalizations of inequality, we refer the interested readers to \cite{mmf, mmf1, mgh, 2}. One of our main aim
in this article is to improve the first inequality in the above. For this purpose, we use some ideas from \cite[Theorem 3.4]{1}.

In this paper, we first provide a better estimate of the second inequality in Hermite-Hadamard inequality.
Our results are employed to establish a new bound for the operator Kantorovich inequality. In particular,
we show that

\begin{footnotesize} \begin{align}\Phi(A^{-1})\le \Phi\Big(\int_0^1\Big(\frac{M\mathbf{1}_{\mathcal{H}}
-A}{M-m}m^{-t}+\frac{A-m\mathbf{1}_{\mathcal{H}}}{M-m}M^{-t}\Big)^{\frac{1}{t}}dt\Big)\le \frac{(M+m)^2}{4Mm}
\Phi^{-1}(A),\end{align}\end{footnotesize} where $m{{\mathbf{1}}_{\mathcal{H}}}\le A\le M{{\mathbf{1}}_{\mathcal{H}}}$
and $\Phi$ is a positive linear map.

\section{Main Results}

First we start our work by providing a better estimate for H-H inequality (\ref{eq1}).
\begin{theorem}\label{10}
Let $f:J\to (0, \infty)$ be a continuous  function on the interval $J$ such that $f^t$ is convex for all $0<t<1$.  Then for any
$a,b \in J$,
\begin{small}\begin{align*} f\big(\frac{a+b}{2}\big)&\le\frac{1}{b-a}\int_a^b f(z)dz \le \int_a^b\Big(\frac{z-a}{b-a}f^t(a)+
\frac{b-z}{b-a} f^t(b)\Big)^{\frac{1}{t}}dz \le\frac{f(a)+f(b)}{2}.
\end{align*}\end{small}
\end{theorem}
\begin{proof}
Since ${{f}^{t}}$ is a convex function, we have for any $x,y\in J$ and $v\in \left[ 0,1 \right]$
\begin{equation}\label{1}
{{f}^{t}}\left( \left( 1-v \right)x+vy \right)\le \left( 1-v \right){{f}^{t}}\left( x \right)+v{{f}^{t}}\left( y \right).
\end{equation}
Raising both sides of \eqref{1} to the power of ${1}/{t}\;$, we get
	\[\begin{aligned}
   f\left( \left( 1-v \right)x+vy \right)&={{\left( {{f}^{t}}\left( \left( 1-v \right)x+vy \right) \right)}^{\frac{1}{t}}} \\
 & \le {{\left( \left( 1-v \right){{f}^{t}}\left( x \right)+v{{f}^{t}}\left( y \right) \right)}^{\frac{1}{t}}} \\
 & \le \left( 1-v \right){{\left( {{f}^{t}}\left( x \right) \right)}^{\frac{1}{t}}}+v{{\left( {{f}^{t}}\left( y \right) \right)}^{\frac{1}{t}}} \\
 & =\left( 1-v \right)f\left( x \right)+vf\left( y \right).
\end{aligned}\]
Consequently,
\begin{equation}\label{2}
f\left( \left( 1-v \right)x+vy \right)\le {{\left( \left( 1-v \right){{f}^{t}}\left( x \right)+v{{f}^{t}}\left( y \right) \right)}^{\frac{1}{t}}}\le \left( 1-v \right)f\left( x \right)+vf\left( y \right),
\end{equation}
which shows the convexity of the function $f$.

Suppose  $z\in \left[ a,b \right]$. If we substitute $x=a$, $y=b$, and $1-v={\left( b-z \right)}/{\left( b-a \right)}\;$ in \eqref{2}, we get
\begin{equation}\label{3}
\begin{aligned}
   f\left( z \right)&\le {{\left( \frac{b-z}{b-a}{{f}^{t}}\left( a \right)+\frac{z-a}{b-a}{{f}^{t}}\left( b \right) \right)}^{\frac{1}{t}}} \\
 & \le \frac{b-z}{b-a}f\left( a \right)+\frac{z-a}{b-a}f\left( b \right).
\end{aligned}
\end{equation}
Since $z\in \left[ a,b \right]$, it follows that $b+a-z\in \left[ a,b \right]$. Now, applying the inequality \eqref{3} to the variable $b+a-z$, we get
\begin{equation}\label{4}
\begin{aligned}
   f\left( b+a-z \right)&\le {{\left( \frac{z-a}{b-a}{{f}^{t}}\left( a \right)+\frac{b-z}{b-a}{{f}^{t}}\left( b \right) \right)}^{\frac{1}{t}}} \\
 & \le \frac{z-a}{b-a}f\left( a \right)+\frac{b-z}{b-a}f\left( b \right).
\end{aligned}
\end{equation}
By adding inequalities \eqref{3} and \eqref{4}, we infer that
\[\begin{aligned}
  & f\left( b+a-z \right)+f\left( z \right) \\
 & \le {{\left( \frac{z-a}{b-a}{{f}^{t}}\left( a \right)+\frac{b-z}{b-a}{{f}^{t}}\left( b \right) \right)}^{\frac{1}{t}}}+{{\left( \frac{b-z}{b-a}{{f}^{t}}\left( a \right)+\frac{z-a}{b-a}{{f}^{t}}\left( b \right) \right)}^{\frac{1}{t}}} \\
 & \le \frac{z-a}{b-a}f\left( a \right)+\frac{b-z}{b-a}f\left( b \right)+\frac{b-z}{b-a}f\left( a \right)+\frac{z-a}{b-a}f\left( b \right) \\
 & =f\left( b \right)+f\left( a \right)
\end{aligned}\]
which, in turn, leads to
\begin{equation}\label{5}
\begin{aligned}
  & f\left( \frac{a+b}{2} \right) \\
 & \le \frac{f\left( a+b-z \right)+f\left( z \right)}{2} \\
 & \le \frac{1}{2}\left( {{\left( \frac{z-a}{b-a}{{f}^{t}}\left( a \right)+\frac{b-z}{b-a}{{f}^{t}}\left( b \right) \right)}^{\frac{1}{t}}}+{{\left( \frac{b-z}{b-a}{{f}^{t}}\left( a \right)+\frac{z-a}{b-a}{{f}^{t}}\left( b \right) \right)}^{\frac{1}{t}}} \right) \\
 & \le \frac{f\left( a \right)+f\left( b \right)}{2}.
\end{aligned}
\end{equation}
Now, the result follows by integrating the inequality \eqref{5} over $z\in \left[ a,b \right]$, and using the fact that $\int_{a}^{b}{f\left( z \right)dz}=\int_{a}^{b}{f\left( a+b-z \right)dz}$.
\end{proof}

\begin{remark}
It follows from the proof of Theorem \ref{10} that
\[\begin{aligned}
  & f\left( a+b-z \right) \\
 & \le {{\left( \frac{z-a}{b-a}{{f}^{t}}\left( a \right)+\frac{b-z}{b-a}{{f}^{t}}\left( b \right) \right)}^{\frac{1}{t}}}+{{\left( \frac{b-z}{b-a}{{f}^{t}}\left( a \right)+\frac{z-a}{b-a}{{f}^{t}}\left( b \right) \right)}^{\frac{1}{t}}}-f\left( z \right) \\
 & \le f\left( a \right)+f\left( b \right)-f\left( z \right).
\end{aligned}\]

This inequality improves \cite[Lemma 1.3]{3}.
\end{remark}

We introduce two notations that will be used in the sequel:
\[{{a}_{f}}\equiv \frac{f\left( M \right)-f\left( m \right)}{M-m}\text{ }\And \text{ }{{b}_{f}}\equiv \frac{Mf\left( m \right)-mf\left( M \right)}{M-m}.\]
The following result gives us a refinement of the first inequality in inequality (\ref{eq2}).

\begin{theorem}\label{6}
Let $A\in \mathcal{B}\left( \mathcal{H} \right)$ be a self-adjoint operator with the spectra contained in the interval $\left[ m,M \right]$ with $m<M$, and let $\Phi :\mathcal{B}\left( \mathcal{H} \right)\to \mathcal{B}\left( \mathcal{K} \right)$ be a positive unital linear
map. If $f:\left[ m,M \right]\to \left( 0,\infty  \right)$ is a continuous function  such that ${{f}^{t}}$ $t\in \left( 0,1 \right)$ is convex, then for a given real number $\alpha$
\[\begin{aligned}
   \Phi(f(A))&\le \Phi \left( {{\left( \frac{M{{\mathbf{1}}_{\mathcal{H}}}-A}{M-m}{{f}^{t}}\left( m \right)+\frac{A-m{{\mathbf{1}}_{\mathcal{H}}}}{M-m}{{f}^{t}}\left( M \right) \right)}^{\frac{1}{t}}} \right) \\
 & \le \beta {{\mathbf{1}}_{\mathcal{K}}}+\alpha f\left( \Phi \left( A \right) \right)
\end{aligned}\]
holds for $\beta =\underset{m\le t\le M}{\mathop{\max }}\,\left\{ {{a}_{f}}t+{{b}_{f}}-\alpha f\left( t \right) \right\}$.
\end{theorem}
\begin{proof}
We first observe that the assumptions imply that for any $m\le z\le M$,
\begin{equation}\label{11}
\begin{aligned} f(z)&\le {{\left( \frac{M-z}{M-m}{{f}^{t}}\left( m \right)+\frac{z-m}{M-m}{{f}^{t}}\left( M \right)
\right)}^{\frac{1}{t}}} \\  & \le \frac{M-z}{M-m}f\left( m \right)+\frac{z-m}{M-m}f\left( M \right).
\end{aligned}
\end{equation}
Applying the continuous functional calculus for the operator $A$  whose spectrum is contained in the interval $\left[ m,M \right]$,
\[\begin{aligned}
   f\left( A \right)&\le {{\left( \frac{M{{\mathbf{1}}_{\mathcal{H}}}-A}{M-m}{{f}^{t}}\left( m \right)+\frac{A-m{{\mathbf{1}}_{\mathcal{H}}}}{M-m}{{f}^{t}}\left( M \right) \right)}^{\frac{1}{t}}} \\
 & \le \frac{M{{\mathbf{1}}_{\mathcal{H}}}-A}{M-m}f\left( m \right)+\frac{A-m{{\mathbf{1}}_{\mathcal{H}}}}{M-m}f\left( M \right).
\end{aligned}\]
Since $\Phi$ is order preserving, we have
\[\begin{aligned}
   \Phi \left( f\left( A \right) \right)&\le \Phi \left( {{\left( \frac{M{{\mathbf{1}}_{\mathcal{H}}}-A}{M-m}{{f}^{t}}\left( m \right)+\frac{A-m{{\mathbf{1}}_{\mathcal{H}}}}{M-m}{{f}^{t}}\left( M \right) \right)}^{\frac{1}{t}}} \right) \\
 & =\Phi \left( \frac{M{{\mathbf{1}}_{\mathcal{H}}}-A}{M-m}f\left( m \right)+\frac{A-m{{1}_{H}}}{M-m}f\left( M \right) \right) \\
 & \le \frac{M{{\mathbf{1}}_{\mathcal{K}}}-\Phi \left( A \right)}{M-m}f\left( m \right)+\frac{\Phi \left( A \right)-m{{\mathbf{1}}_{\mathcal{K}}}}{M-m}f\left( M \right).
\end{aligned}\]
The above inequality can also be written as
\[\begin{aligned}
   \Phi \left( f\left( A \right) \right)&\le \Phi \left( {{\left( \frac{M{{\mathbf{1}}_{\mathcal{H}}}-A}{M-m}{{f}^{t}}\left( m \right)+\frac{A-m{{\mathbf{1}}_{\mathcal{H}}}}{M-m}{{f}^{t}}\left( M \right) \right)}^{\frac{1}{t}}} \right) \\
 & \le {{a}_{f}}\Phi \left( A \right)+{{b}_{f}}{{\mathbf{1}}_{\mathcal{K}}}.
\end{aligned}\]
Therefore,
\[\begin{aligned}
  & \Phi \left( f\left( A \right) \right)-\alpha f\left( \Phi \left( A \right) \right) \\
 & \le \Phi \left( {{\left( \frac{M{{\mathbf{1}}_{\mathcal{H}}}-A}{M-m}{{f}^{t}}\left( m \right)+\frac{A-m{{\mathbf{1}}_{\mathcal{H}}}}{M-m}{{f}^{t}}\left( M \right) \right)}^{\frac{1}{t}}} \right)-\alpha f\left( \Phi \left( A \right) \right) \\
 & \le {{a}_{f}}\Phi \left( A \right)+{{b}_{f}}{{\mathbf{1}}_{\mathcal{K}}}-\alpha f\left( \Phi \left( A \right) \right) \\
 & \le \underset{m\le t\le M}{\mathop{\max }}\,\left\{ {{a}_{f}}t+{{b}_{f}}-\alpha f\left( t \right) \right\}  \mathbf{1}_\mathcal{K}.
\end{aligned}\]
Consequently,
\[\begin{aligned}
   \Phi \left( f\left( A \right) \right)&\le \Phi \left( {{\left( \frac{M{{\mathbf{1}}_{\mathcal{H}}}-A}{M-m}{{f}^{t}}\left( m \right)+\frac{A-m{{\mathbf{1}}_{\mathcal{H}}}}{M-m}{{f}^{t}}\left( M \right) \right)}^{\frac{1}{t}}} \right) \\
 & \le \underset{m\le t\le M}{\mathop{\max }}\,\left\{ {{a}_{f}}t+{{b}_{f}}-\alpha f\left( t \right) \right\}\mathbf{1}_\mathcal{K}+\alpha f\left( \Phi \left( A \right) \right)
\end{aligned}\]
and this concludes the proof.
\end{proof}

We have Corollary \ref{7} if we put $\alpha =1$ in Theorem \ref{6} and Corollary \ref{8} if we
choose $\alpha$ such that $\beta =0$ in Theorem \ref{6}.
\begin{corollary}\label{7}
Let the hypothesis of Theorem \ref{6} be satisfied. Then
\[\begin{aligned}
   \Phi \left( f\left( A \right) \right)&\le \Phi \left( {{\left( \frac{M{{\mathbf{1}}_{\mathcal{H}}}-A}{M-m}{{f}^{t}}\left( m \right)+\frac{A-m{{\mathbf{1}}_{\mathcal{H}}}}{M-m}{{f}^{t}}\left( M \right) \right)}^{\frac{1}{t}}} \right) \\
 & \le \beta {{\mathbf{1}}_{\mathcal{K}}}+f\left( \Phi \left( A \right) \right),
\end{aligned}\]
where
\[\beta =\underset{m\le t\le M}{\mathop{\max }}\,\left\{ \frac{M-t}{M+m}f\left( m \right)+\frac{t-m}{M+m}f\left( M \right)-f\left( t \right) \right\}.\]
\end{corollary}

\begin{corollary}\label{8}
Let the hypothesis of Theorem \ref{6} be satisfied. Then
\[\begin{aligned}
   \Phi \left( f\left( A \right) \right)&\le \Phi \left( {{\left( \frac{M{{\mathbf{1}}_{\mathcal{H}}}-A}{M-m}{{f}^{t}}\left( m \right)+\frac{A-m{{\mathbf{1}}_{\mathcal{H}}}}{M-m}{{f}^{t}}\left( M \right) \right)}^{\frac{1}{t}}} \right) \\
 & \le \alpha f\left( \Phi \left( A \right) \right),
\end{aligned}\]
where
\[\alpha =\underset{m\le t\le M}{\mathop{\max }}\,\left\{ \frac{1}{f\left( t \right)}\left( \frac{M-t}{M+m}f\left( m \right)+\frac{t-m}{M+m}f\left( M \right) \right) \right\}.\]
\end{corollary}

Let $0<t<1$ and let $1<\frac{1}{t}\le r$. Consider the function $f\left( t \right)={{t}^{r}}$. Then we have the following two corollaries.
\begin{corollary}\label{9}
Let $A\in \mathcal{B}\left( \mathcal{H} \right)$ be a self-adjoint operator with the spectra contained in the interval $\left[ m,M \right]$ with $0<m<M$, and let $\Phi :\mathcal{B}\left( \mathcal{H} \right)\to \mathcal{B}\left( \mathcal{K} \right)$ be a positive unital linear map. Then for any $1<\frac{1}{t}\le r$, $t\in \left( 0,1 \right)$
\[\begin{aligned}
   \Phi \left( {{A}^{r}} \right)&\le \Phi \left( {{\left( \frac{M{{\mathbf{1}}_{\mathcal{H}}}-A}{M-m}{{m}^{tr}}+\frac{A-m{{\mathbf{1}}_{\mathcal{H}}}}{M-m}{{M}^{tr}} \right)}^{\frac{1}{t}}} \right) \\
 & \le K\left( m,M,r \right){{\Phi }^{r}}\left( A \right),
\end{aligned}\]
where the generalized Kantorovich constant $K\left( m,M,r \right)$ (\cite[Definition 2.2]{book}) is defined by
\begin{equation}\label{12}
K\left( m,M,r \right)=\frac{\left( m{{M}^{r}}-M{{m}^{r}} \right)}{\left( r-1 \right)\left( M-m \right)}{{\left( \frac{r-1}{r}\frac{{{M}^{r}}-{{m}^{r}}}{m{{M}^{r}}-M{{m}^{r}}} \right)}^{r}}.
\end{equation}
In particular,
\[\begin{aligned}
   \Phi \left( {{A}^{r}} \right)&\le \Phi \left( {{\left( \frac{M{{\mathbf{1}}_{\mathcal{H}}}-A}{M-m}{{m}^{\frac{r}{2}}}+\frac{A-m{{\mathbf{1}}_{\mathcal{H}}}}{M-m}{{M}^{\frac{r}{2}}} \right)}^{2}} \right) \\
 & \le K\left( m,M,r \right){{\Phi }^{r}}\left( A \right),
\end{aligned}\]
for any $r\ge 2$.\par
The above inequalities also hold when $r<0$.
\end{corollary}

\begin{remark}
It follows from Corollary \ref{9} that
\begin{equation}\label{16}
\Phi \left( {{A}^{-1}} \right)\le \Phi \left( {{\left( \frac{M{{\mathbf{1}}_{\mathcal{H}}}-A}{M-m}{{m}^{-t}}+\frac{A-m{{\mathbf{1}}_{\mathcal{H}}}}{M-m}{{M}^{-t}} \right)}^{\frac{1}{t}}} \right)\le \frac{{{\left( M+m \right)}^{2}}}{4Mm}{{\Phi }^{-1}}\left( A \right),
\end{equation}
since $K\left( m,M,-1 \right)=\frac{{{\left( M+m \right)}^{2}}}{4Mm}$. Integrating the inequality \eqref{9} over $t\in \left[ 0,1 \right]$, we find that
	\[\Phi \left( {{A}^{-1}} \right)\le \int_{0}^{1}{\Phi \left( {{\left( \frac{M{{\mathbf{1}}_{\mathcal{H}}}-A}{M-m}{{m}^{-t}}+\frac{A-m{{\mathbf{1}}_{\mathcal{H}}}}{M-m}{{M}^{-t}} \right)}^{\frac{1}{t}}} \right)dt}\le \frac{{{\left( M+m \right)}^{2}}}{4Mm}{{\Phi }^{-1}}\left( A \right).\]
Since the mapping $\Phi $ is linear and continuous, then
\begin{equation}\label{17}
\Phi \left( {{A}^{-1}} \right)\le \Phi \left( \int_{0}^{1}{{{\left( \frac{M{{\mathbf{1}}_{\mathcal{H}}}-A}{M-m}{{m}^{-t}}+\frac{A-m{{\mathbf{1}}_{\mathcal{H}}}}{M-m}{{M}^{-t}} \right)}^{\frac{1}{t}}}dt} \right)\le \frac{{{\left( M+m \right)}^{2}}}{4Mm}{{\Phi }^{-1}}\left( A \right).
\end{equation}
Observe that the inequality \eqref{17} gives a refinement of the operator Kantorovich inequality \cite{kantorovich}.
\end{remark}

\begin{corollary}
Let the hypothesis of Corollary \ref{8} be satisfied. Then
\[\begin{aligned}
   \Phi \left( {{A}^{r}} \right)&\le \Phi \left( {{\left( \frac{M{{\mathbf{1}}_{\mathcal{H}}}-A}{M-m}{{m}^{tr}}+\frac{A-m{{\mathbf{1}}_{\mathcal{H}}}}{M-m}{{M}^{tr}} \right)}^{\frac{1}{t}}} \right) \\
 & \le C\left( m,M,r \right){{\mathbf{1}}_{\mathcal{K}}}+{{\Phi }^{r}}\left( A \right),
\end{aligned}\]
where the Kantorovich constant for the difference $C\left( m,M,r \right)$ (\cite[Theorem 2.58]{book}) is
defined by
\begin{equation}\label{13}
C\left( m,M,r \right)=\frac{M{{m}^{r}}-m{{M}^{r}}}{M-m}+\left( r-1 \right){{\left( \frac{{{M}^{r}}-{{m}^{r}}}{r\left( M-m \right)} \right)}^{\frac{r}{r-1}}}.
\end{equation}
In particular,
\[\begin{aligned}
   \Phi \left( {{A}^{r}} \right)&\le \Phi \left( {{\left( \frac{M{{\mathbf{1}}_{\mathcal{H}}}-A}{M-m}{{m}^{\frac{r}{2}}}+\frac{A-m{{\mathbf{1}}_{\mathcal{H}}}}{M-m}{{M}^{\frac{r}{2}}} \right)}^{2}} \right) \\
 & \le C\left( m,M,r \right){{\mathbf{1}}_{\mathcal{K}}}+{{\Phi }^{r}}\left( A \right)
\end{aligned}\]
for any $r\ge 2$.\par
The above inequalities also hold when $r<0$.
\end{corollary}

\begin{example}
Letting $t={1}/{2}\;$ and $\Phi \left( T \right)=\frac{1}{2}tr\left( T \right)$. Consider $A=\left[ \begin{matrix}
   2 & -1  \\
   -1 & 3  \\
\end{matrix} \right]$. Then, of course, we can choose $m=1.35$ and $M=3.8$. Simple calculations show that,
	\[\Phi \left( {{A}^{-1}} \right)=0.5,\]
	\[\Phi \left( {{\left( \frac{M\mathbf{1}_\mathcal{H}-A}{M-m}\frac{1}{\sqrt{m}}+\frac{A-m\mathbf{1}_\mathcal{H}}{M-m}\frac{1}{\sqrt{M}} \right)}^{2}} \right)\approx 0.51,\]
	\[\frac{{{\left( M+m \right)}^{2}}}{4Mm}{{\Phi }^{-1}}\left( A \right)\approx 0.517.\]
Consequently,
\[\begin{aligned}
   \Phi \left( {{A}^{-1}} \right)&\lneqq \Phi \left( {{\left( \frac{M\mathbf{1}_\mathcal{H}-A}{M-m}\frac{1}{\sqrt{m}}+\frac{A-m\mathbf{1}_\mathcal{H}}{M-m}\frac{1}{\sqrt{M}} \right)}^{2}} \right) \\
 & \lneqq \frac{{{\left( M+m \right)}^{2}}}{4Mm}{{\Phi }^{-1}}\left( A \right).
\end{aligned}\]
\end{example}

Remind that the function $f\left( t \right)={{t}^{r}}$ for $r>1$ is not operator monotone on $\left[ 0,\infty  \right)$. In the sense that $A\le B$ does not always ensure ${{A}^{r}}\le {{B}^{r}}$. Related to this problem, Furuta \cite{n3} proved: 	Let $A,B\in \mathbb{B}\left( \mathcal{H} \right)$ be two positive operators such that their spectrums contained in the interval $\left[ m,M \right]$, for some scalars $0<m<M$. If $A\le B$, then
\begin{equation}\label{15}
{{A}^{p}}\le {{K}}\left( m,M,r \right){{B}^{r}}\quad\text{ for }r\ge 1.
\end{equation}
Next, we present a better estimate than Furuta inequality \eqref{15}. To this end, we recall the following operator version of Jensen's inequality which is shown by Mond and Pe\v cari\'c in \cite[Theorem 1.2]{book}.
\begin{lemma}
Let $A\in \mathcal{B}\left( \mathcal{H} \right)$ be a self-adjoint operator with the spectra contained in the interval $J$ and let $x\in \mathcal{H}$. If $f:J\to \mathbb{R}$ is convex function, then
\[f\left( \left\langle Ax,x \right\rangle  \right)\le \left\langle f\left( A \right)x,x \right\rangle.\]
\end{lemma}

\begin{theorem}\label{14}
Let $A,B\in \mathcal{B}\left( \mathcal{H} \right)$ be two self-adjoint operators with the spectra contained in the interval $\left[ m,M \right]$ with $m<M$, and let $\Phi :\mathcal{B}\left( \mathcal{H} \right)\to \mathcal{B}\left( \mathcal{K} \right)$ be a positive unital linear map. If $f:\left[ m,M \right]\to \left( 0,\infty  \right)$ is a continuous increasing function  such that ${{f}^{t}}$ $t\in \left( 0,1 \right)$ is convex, then for a given positive real number $\alpha$
\[\begin{aligned}
   f\left( A \right)&\le {{\left( \frac{M\mathbf{1}_\mathcal{H}-A}{M-m}{{f}^{t}}\left( m \right)+\frac{A-m\mathbf{1}_\mathcal{H}}{M-m}{{f}^{t}}\left( M \right) \right)}^{\frac{1}{t}}} \\
 & \le \beta {{\mathbf{1}}_{\mathcal{H}}}+\alpha f\left( B \right),
\end{aligned}\]
holds for $\beta =\underset{m\le t\le M}{\mathop{\max }}\,\left\{ {{a}_{f}}t+{{b}_{f}}-\alpha f\left( t \right) \right\}$.
\end{theorem}

\begin{proof}
Our assumption implies
\[\begin{aligned}
  & f\left( A \right)\le {{\left( \frac{M\mathbf{1}_\mathcal{H}-A}{M-m}{{f}^{t}}\left( m \right)+\frac{A-m\mathbf{1}_\mathcal{H}}{M-m}{{f}^{t}}\left( M \right) \right)}^{\frac{1}{t}}} \\
 & \le {{a}_{f}}A+{{b}_{f}}\mathbf{1}_\mathcal{H}.
\end{aligned}\]
thanks to \eqref{11}. Then, for any unit vector $x \in \mathcal{H}$
\[\begin{aligned}
  & \left\langle f\left( A \right)x,x \right\rangle  \\
 & \le \left\langle {{\left( \frac{M\mathbf{1}_\mathcal{H}-A}{M-m}{{f}^{t}}\left( m \right)+\frac{A-m\mathbf{1}_\mathcal{H}}{M-m}{{f}^{t}}\left( M \right) \right)}^{\frac{1}{t}}}x,x \right\rangle  \\
 & \le {{a}_{f}}\left\langle Ax,x \right\rangle +{{b}_{f}}.
\end{aligned}\]
Therefore,
\[\begin{aligned}
  & \left\langle f\left( A \right)x,x \right\rangle -\alpha f\left( \left\langle Bx,x \right\rangle  \right) \\
 & \le \left\langle {{\left( \frac{M\mathbf{1}_\mathcal{H}-A}{M-m}{{f}^{t}}\left( m \right)+\frac{A-m\mathbf{1}_\mathcal{H}}{M-m}{{f}^{t}}\left( M \right) \right)}^{\frac{1}{t}}}x,x \right\rangle -\alpha f\left( \left\langle Bx,x \right\rangle  \right) \\
 & \le {{a}_{f}}\left\langle Ax,x \right\rangle +{{b}_{f}}-\alpha f\left( \left\langle Bx,x \right\rangle  \right) \\
 & \le {{a}_{f}}\left\langle Ax,x \right\rangle +{{b}_{f}}-\alpha f\left( \left\langle Ax,x \right\rangle  \right) \\
 & \le \underset{m\le t\le M}{\mathop{\max }}\,\left\{ {{a}_{f}}t+{{b}_{f}}-\alpha f\left( t \right) \right\}.
\end{aligned}\]
On the other hand, since $A\le B$ and $f$ is increasing and convex, then
\[\begin{aligned}
   \left\langle f\left( A \right)x,x \right\rangle &\le \left\langle {{\left( \frac{M\mathbf{1}_\mathcal{H}-A}{M-m}{{f}^{t}}\left( m \right)+\frac{A-m\mathbf{1}_\mathcal{H}}{M-m}{{f}^{t}}\left( M \right) \right)}^{\frac{1}{t}}}x,x \right\rangle  \\
 & \le \beta +\alpha f\left( \left\langle Bx,x \right\rangle  \right) \\
 & \le \beta +\alpha \left\langle f\left( B \right)x,x \right\rangle  \\
 & =\left\langle \left( \beta {{\mathbf{1}}_{\mathcal{H}}}+\alpha f\left( B \right) \right)x,x \right\rangle.
\end{aligned}\]
This completes the proof.
\end{proof}

As a direct consequence of Theorem \ref{14}, we have:
\begin{corollary}
Let $A,B\in \mathcal{B}\left( \mathcal{H} \right)$ be two self-adjoint operators with the spectra contained in the interval $\left[ m,M \right]$ with $0<m<M$. If $A\le B$, then for any $1<\frac{1}{t}\le r$, $t\in \left( 0,1 \right)$
	\[\begin{aligned}
   {{A}^{r}}&\le {{\left( \frac{M\mathbf{1}_\mathcal{H}-A}{M-m}{{m}^{tr}}+\frac{A-m\mathbf{1}_\mathcal{H}}{M-m}{{M}^{tr}} \right)}^{\frac{1}{t}}} \\
 & \le K\left( m,M,r \right){{B}^{r}}
\end{aligned}\]
and
	\[\begin{aligned}
   {{A}^{r}}&\le {{\left( \frac{M\mathbf{1}_\mathcal{H}-A}{M-m}{{m}^{tr}}+\frac{A-m\mathbf{1}_\mathcal{H}}{M-m}{{M}^{tr}} \right)}^{\frac{1}{t}}} \\
 & \le C\left( m,M,r \right)\mathbf{1}_\mathcal{H}+{{B}^{r}} \\
\end{aligned}\]
where $K\left( m,M,r \right)$ and $C\left( m,M,r \right)$ are defined as in \eqref{12} and \eqref{13}, respectively.
\end{corollary}

\section*{Acknowledgements}
The authors would like to thank the anonymous referees for their comments and suggestions on preliminary versions of this
paper, which have led to a substantial improvement in its readability.
This work was partially supported by the Islamic Azad University, Central Tehran Branch.

\end{document}